\documentclass[a4paper,11pt]{amsart}
\usepackage{amsfonts,amssymb,epsfig,latexsym}
\usepackage{amsmath}
\usepackage[latin1]{inputenc}
\usepackage{color,layout}
\usepackage{ifthen}

\usepackage{dsfont}

\usepackage{pdfsync}
\usepackage{graphicx}
\usepackage{color}
\usepackage{pdfsync}
\date{}

\newcommand{\be}{\begin{equation}}
\newcommand{\ee}{\end{equation}}

\def\la{\langle}
\def\ra{\rangle}

\def\R{\mathbb{R}}
\def\C{\mathbb{C}}

\def\N{\mathbb{N}}

\renewcommand{\Re}{\text{{\rm Re}\;}}
\renewcommand{\Im}{\text{{\rm Im}\;}}

\newtheorem{theorem}{Theorem}[section]
\newtheorem{lemma}[theorem]{Lemma}
\newtheorem{proposition}[theorem]{Proposition}
\newtheorem{corollary}[theorem]{Corollary}

\theoremstyle{definition}

\newtheorem{remark}[theorem]{Remark}
\numberwithin{equation}{section}

\title[Molecular dynamics for predissociation]{Estimates on the molecular dynamics for the predissociation process}
\author[Philippe BRIET \& 
Andr\'e MARTINEZ]{
Philippe BRIET${}^1$ \& 
Andr\'e MARTINEZ$ {}^2$
 }

\begin{document}

\maketitle 
\addtocounter{footnote}{1}
\footnotetext{{\tt\small  Aix-Marseille Universit\'{e}, CNRS, CPT UMR 7332, 13288 Marseille, France, and Universit\'{e} de Toulon, CNRS, CPT UMR 7332, 83957 La Garde, France,  
briet@univ-tln.fr} }  
\addtocounter{footnote}{1}
\footnotetext{{\tt\small Universit\`a di Bologna,  
Dipartimento di Matematica, Piazza di Porta San Donato, 40127
Bologna, Italy, 
andre.martinez@unibo.it }}  
\begin{abstract}
We study the survival probability associated with a semiclassical matrix Schr\"odinger operator that models the predissociation of a general molecule in the Born-Oppenheimer approximation. We show that it is given by its usual time-dependent exponential contribution, up to a reminder term that is small in the semiclassical parameter and for which we find the main contribution. The result applies in any dimension, and in presence of a number of resonances that may tend to infinity as the semiclassical parameter tends to 0.
\end{abstract}  

\baselineskip = 18pt 
\vfill\eject
\section{Introduction}

The molecular predissociation is  one of most well  known quantum phenomena giving rise to  metastable states and resonances. This corresponds    when  a  bound state molecule
dissociates to the continuum through tunneling  see e.g. \cite{Kr, La, St, Ze}.
The rigorous description of this phenomena goes back  to \cite{Kl} with further developments in \cite{DuMe} and, more recently, in \cite{GrMa}.

In the context of  the Born-Oppenheimer approximation, the transition   can occur  when a confining electronic   curve near  a given energy  $E$ (e.g. $E$ is  a local minimum)  crosses a dissociative electronic  level (that is, a  curve having a   limit smaller than $E$  at infinity). Such a situation occurs for instance in the $SH$ molecule : see \cite{LeSu}. 

After reduction to an effective Hamiltonian, this phenomena can be described by a $2\times 2$ matrix $H$ of semiclassical pseudodifferential operators (see, e.g., \cite{KMSW, MaSo}), with small parameter $h$ corresponding to the square root of the inverse of the mass of the nuclei, and with principal part that is diagonal and consists of two Schr\"odinger operators. 

In this paper we consider predissociation resonances   from a dynamical point of view, i.e. in terms of exponential behavior in time  of the  quantum evolution $e^{-itH}$
associated with that system. 

Our main motivation is the recent series of works around the case where $H=H_0+\kappa V$ is the perturbation of an operator with an embedded eigenvalue: See, e.g., \cite{CGH, CoSo, JeNe, Her, Hu2} and references therein. In all of these papers, denoting by $\varphi$ the corresponding eigenfunction of $H_0$, the survival probability $\la e^{-itH}\varphi, \varphi \ra$ is studied. Roughly speaking, they show that the embedded eigenvalue gives rise to a resonance $\rho$, and  the previous quantity behaves like $e^{-it\rho} \|\varphi\|^2$ with an error-term typically ${\mathcal O}(\kappa^2)$. Moreover, inserting a cutoff in energy, the error-term has  a polynomial decay in time at infinity.

The starting point of our work is the following observation: in the case of the molecular predissociation, $H$ can be seen as a perturbation of a matrix Schr\"odinger operator admitting embedded eigenvalues. Therefore, a similar procedure can be done in order to study the quantum evolution. However, in contrast with the case $H=H_0+\kappa V$, the small parameter is involved in the unperturbed operator, too, making very delicate the extension of the methods used for it. In order to overcome this difficulty, we use the definition of resonances based of complex distortion (see, e.g., \cite{Hu1}), and we replace the arguments of regular perturbation theory (used, e.g., in \cite{CGH}) by those of semiclassical microlocal analysis.  

In this way, we can essentially generalize the previous results, and in the case of an isolated resonance $\rho$ with a gap $a(h)>>h^2$, our result takes the form,
$$
\la e^{-itH}g(H)\varphi, \varphi\ra =e^{-it\rho}b(\varphi, h) +{\mathcal O}\left(  \frac{h^2}{a(h)} \min_{0\leq k \leq \nu }\{ \frac{1}{[(1+t)a(h)]^{k}}\}\|\varphi\|^2\right),
$$
where  $b(\varphi,h)$ is the residue at $\rho$ of $z\mapsto \la (z-H)^{-1}\varphi, \varphi\ra$, and $\nu\geq 0$ depends on the regularity of the energy cutoff $g$ (see Theorem \ref{mainth} for a more complete result with several resonances). In addition, we also have an expression for the main contribution of the remainder term (see Remark \ref{remr0}). In the case where $\nu$ can be taken positive, this also leads to the fact that the error term remains negligible up to times of order $Ch^{-1}|\Im\rho|^{-1}$ with $C>0$, $C\sim\nu$ as $\nu\to\infty$, that is much beyond the life-time of the resonant state (see Remark \ref{rembeyondlt}).

Our results must also be compared with that of \cite{NSZ}, where a polynomial bound is obtained for the rest in the quantum evolution, in the case of a scalar semiclassical Schr\"odinger operator.

Let us briefly describe the content of the paper. In the next section, we give a precise description of the model and assumptions. Section \ref{secRes} is devoted to the definition of resonances by means of complex distortion theory. Our main result is given in Section \ref{secmainth}, whose proof is spread over Sections \ref{secprelim} to \ref{secRest}. Section \ref{secCor} contains the proof of a corollary where the energy cutoff has been removed and we discuss in Section \ref{casNT} the non-trapping case. Finally, some examples of application are given in Section \ref{secexample}, and the Appendix contains the proof of some technical results.

\section{Assumptions}

We consider the semiclassical $2\times 2$ matrix Schr\"odinger operator,
\be
\label{operator}
H= H_0+ h{\mathcal W}(x,hD_x)=
\left(\begin{array}{cc}
P_1 & 0\\
0 & P_2
\end{array}\right) + h{\mathcal W}(x,hD_x)
\ee
on the Hilbert space ${\mathcal H}:=L^2(\R^n)\oplus L^2(\R^n)$, with,
$$
P_j := -h^2\Delta +V_j(x) \quad (j=1,2),
$$
where $x=(x_1,\dots ,x_n)$ is the current variable in $\R^n$ ($n\geq 1$),
$h>0$ denotes the semiclassical parameter, and 
$$
{\mathcal W}(x,hD_x)=
\left(\begin{array}{cc}
0 & W\\
W^* & 0
\end{array}\right)
$$
with $W=w(x,hD_x)$ is a first-order
semiclassical pseudodifferential operators, in the sense that, for all $\alpha\in \N^{2n}$, $\partial^\alpha w(x,\xi) ={\mathcal O}(1+|\xi|)$ uniformly on $\R^{2n}$.

This is typically the kind of operator one obtains in the Born-Oppenheimer approximation, after reduction to an effective Hamiltonian (see \cite{KMSW, MaSo}). 

We assume,

{\bf Assumption 1.} { \it The potentials $V_1$ and $V_2$ are smooth and 
bounded on $\R^n$,   and
satisfy,
\begin{eqnarray}
\label{assV1}
&&\mbox{The set }\, U:=\{V_1\leq 0\} \mbox{ is bounded } ;\\
&& \liminf_{\vert x\vert\rightarrow\infty} V_1 >0;\\
&&V_2 \; \mbox{has a strictly negative limit } -\Gamma \mbox{ as } \; \vert x\vert \rightarrow \infty;\\
&& V_2 >0 \mbox{ on } U.
\end{eqnarray}
}

In particular, $H$ with domain ${\mathcal D}_H:=H^2(\R^n)\oplus H^2(\R^n)$ is selfadjoint.

Since we have to consider the resonances of $H$ near the energy level $E=0$, we also assume,

{\bf Assumption 2.}\label{Ass2} {\it The potentials $V_1$ and $V_2$ extend to bounded holomorphic functions near a complex sector of the form,
${\mathcal S}_{R_0,\delta} :=\{x\in \C^n\, ;\, |\Re x|\geq R_0\, ,\,  \vert \Im x\vert \leq \delta |\Re x| \}$, with $R_0,\delta >0$. Moreover  $V_2$ tends to its limit at $\infty$ in this sector and $\Re V_1 $ stays away from $0$ in this sector.}

{\bf Assumption 3.} \label{Ass3}{\it The  symbol $w(x,\xi)$ of $W$ extends to a holomorphic functions  in $(x,\xi)$ near,
$$
\widetilde {\mathcal S}_{R_0,\delta}:= {\mathcal S}_{R_0,\delta}\times \{ \xi\in \C^n\,; |\Im\xi| \leq \delta\la \Re x \ra\},
$$
and, for real $x$, $w$ is a smooth function of $x$ with values in the set of holomorphic functions of $\xi$ near $\{ |\Im\xi|\leq \delta\}$.
Moreover, we assume that, for any $\alpha\in \N^{2n}$,  it satisfies
\be
\label{assR}
\partial^\alpha w(x,\xi) ={\mathcal O}(\la \Re \xi\ra)\,\, \mbox{ uniformly on } \widetilde {\mathcal S}_{R_0,\delta}\, \cup\, \left(\R^n\times \{ |\Im\xi|\leq\delta\}\right).
\ee }

Under the previous assumption we plan to study the quantum evolution of the operator $P$ given in (\ref{operator}), where ${\mathcal W}$ is defined as
$$
{\mathcal W}:= \left( \begin{array}{cc}
0&{\rm Op}^L_h (w)\\
{\rm Op}^R_h (\overline{w})& 0
\end{array} \right)
$$
where for any symbol $a(x,\xi)$ we use the following quantizations,
\begin{eqnarray*}
&& {\rm Op}^L_h (a)u(x) = \frac{1}{(2 \pi h)^n} \int e^{i(x-y)\xi/h} a(x, \xi) u(y) dy d\xi;\\
&&{\rm Op}^R_h (a)u(x) = \frac{1}{(2 \pi h)^n} \int e^{i(x-y)\xi/h} a(y, \xi) u(y) dy d\xi.
\end{eqnarray*}

Finally, we assume,

{\bf Assumption [V]} (Virial condition) \label{viriel}
$$
2V_2(x)+x\nabla V_2(x)<0\,  \mbox{ on } \{V_2\leq 0\},
$$
or, more generally,

{\bf Assumption [NT]}
$$
E=0 \mbox{ is a non-trapping energy for } V_2.
$$

The fact that 0 is a non-trapping energy for $V_2$ means that, for any $(x,\xi)\in p_2^{-1}(0)$, one has $|\exp tH_{p_2}(x,\xi)|\rightarrow +\infty$ as $t\rightarrow \infty$, where $p_2(x,\xi):=\xi^2+V_2(x)$ is the symbol of $P_2$, and $H_{p_2}:=(\nabla_\xi p_2, -\nabla_x p_2)$ is the Hamilton field of $p_2$. It is equivalent to the existence of a function $G\in C^\infty (\R^{2n};\R)$ supported near $\{ p_2=0\}$ (where $p_2(x,\xi):=\xi^2 + V_2(x)$), and satisfying,
\be
\label{fuite}
H_{p_2}G(x,\xi)>0 \,  \mbox{ on } \{p_2=0\}.
\ee
Note that Assumption [V] is nothing but (\ref{fuite}) with $G(x,\xi) =x\cdot \xi$. Moreover, thanks to Assumption 2, we see that this condition is automatically satisfied for $|x|$ large enough.

\section{Resonances}
\label{secRes}

In the previous situation, the essential spectrum of $H_0$ is $[-\Gamma, +\infty)$.
The resonances of $H$ can be defined by using a complex distortion in the following way: Let $f(x) \in C^\infty (\R^n, \R^n)$ such that $f(x) = 0$ for $\vert x \vert \leq R_0$, $f(x) = x$ for $\vert x\vert$ large enough. For $\theta\not=0$ small enough, we define the distorted operator $H_{\theta}$ as the value at $ \nu = i \theta$ of  the extension to the complex of the  operator  $U_\nu H U_\nu^{-1}$ which is  defined for $\nu$ real, and analytic in $\nu$ for $\nu$ small enough, where we have set,
\be
\label{distorsion}
U_\nu \phi(x) := \det ( 1 + \nu df(x))^{1/2} \phi ( x + \nu f(x)).
\ee 
Since we have a pseudodifferential operator $w(x, hD_x)$, the fact that  $U_\nu H U_\nu^{-1}$ is analytic in $\nu$ is not completely standard but can be done without problem (thanks to Assumption 3). By using the Weyl Perturbation Theorem, one can also see  that there exists $\varepsilon_0 >0$ such that for any $\pm\theta >0$ small enough, the spectrum of $H_\theta$ is discrete in $\{ z\in\C\, ;\, \Re z\in [ -\varepsilon_0, \varepsilon_0 ] ,\, \pm \Im z\geq \mp \varepsilon_0 \theta\}$, and contained in $\{\pm\Im z\leq 0\}$. When $\theta$ is positive, the eigenvalues of $H_\theta$ are called the resonances of $H$ \cite{Hu1, HeSj2, HeMa}, they form a set denoted by ${\rm Res} (H)$  (on the contrary, when $\theta<0$, the eigenvalues of $H_\theta$ are just the complex conjugates of the resonances of $H$, and are called anti-resonances). 

Let us observe that the resonances of $H$ can also be viewed as the poles of the meromorphic extension, from $\{\Im z >0\}$, of some matrix elements of the resolvent $R(z):=(H-z)^{-1}$ (see, e.g., \cite{ReSi, HeMa}).

By adapting techniques of \cite{HeSj1, HeSj2} (see also \cite{Kl, GrMa}), one can prove that, in our situation, the resonances of $H$ near 0 are close to the eigenvalues of the operator
\be
\label{puitsbouche}
\widetilde H:= 
\left(\begin{array}{cc}
-h^2\Delta + V_1 & 0\\
0 & -h^2\Delta + \widetilde V_2
\end{array}\right) + h{\mathcal W}(x,hD_x),
\ee
where $ \widetilde V_2\in C^\infty (\R^n; \R)$ coincides with $V_2$ in $\{ V_2\geq \delta\}$ ($\delta >0$ is fixed arbitrarily small), and is such that $\inf \widetilde V_2 >0$. The precise statement is the following one : Let $I(h)$ be a closed interval included in $(-\varepsilon_0, \varepsilon_0)$, and $a(h)>0$ such that $a(h)\to 0$ as $h\to 0_+$, and, for all $\varepsilon >0$ there exists $C_\varepsilon >0$ satisfying,
\be
a(h)\geq \frac1{C_\varepsilon} e^{-\varepsilon /h};
\ee
\be
\label{gap1}
\sigma (\widetilde H)\cap \left( (I(h)+[-3a(h), 3a(h)])\backslash I(h)\right) =\emptyset,
\ee
for all $h>0$ small enough. Then, there exist two constants $\varepsilon_1, C_0>0$ and a bijection,
$$
\widetilde\beta \, : \, \sigma (\widetilde H)\cap I(h)\,  \to \, {\rm Res} (H)\cap \Omega (h),
$$
where we have set,
$$
\Omega (h):= (I(h)+[-a(h), a(h))+i[-\varepsilon_1, 0],
$$
such that,
$$
\widetilde\beta (\lambda) -\lambda ={\mathcal O}(e^{-C_0/h}),
$$
uniformly as $h\to 0_+$.

In particular, since the eigenvalues of $\widetilde P$ are real, one obtains that, for any $\varepsilon >0$, the resonances $\rho$ in $\Omega (h)$ satisfy,
$$
\Im \rho ={\mathcal O}(e^{-C_0/h}).
$$

In what follows, we will show that, under an additional assumption, these resonances are also closed to the eigenvalues of $P_1$.

\begin{remark}\sl
Actually, under an assumption of analyticity on $W$ slightly stronger that Assumption 3 (see \cite{GrMa}), or if $W$ has a simpler form (see \cite{Kl}), $C_0$ can be taken arbitrarily close to $2d(U, \{ V_2\leq 0\})$, where $d$ stands for the Agmon distance associated with the potential  $\min(V_2,V_1)$, that is, the pseudo-distance associated with the pseudo-metric $\max (0, \min(V_2,V_1))dx^2$ .
\end{remark}

\section{Main Result}
\label{secmainth}
For our purpose, we need to have a stronger gap between $I(h)$ and the rest of the spectrum of $P_1$. Namely, we assume the existence of $a(h)>0$, such that,
\be
\label{gap2}
\begin{aligned}
& \frac{h^2}{a(h)} \to 0\, \mbox{ as } \, h\to 0_+;\\
& \sigma (P_1)\cap \left( (I(h)+[-3a(h), 3a(h)])\backslash I(h)\right) =\emptyset,
\end{aligned}
\ee

\label{secres}
Then, we denote by $u_1,\dots,u_m$ an orthonormal basis of eigenfunctions of $P_1$ corresponding to its eigenvalues $\lambda_1,\dots,\lambda_m$ in $I(h)$ (we recall that $m=m(h) ={\mathcal O}(h^{-n})$). For $j=1,\dots,m$, we also set,
$$
\phi_j:=\left(\begin{array}{c}
u_j\\
0
\end{array}\right) \in L^2(\R^n)\oplus L^2(\R^n),
$$
so that $\phi_j$ is an eigenvector of,
$$
H_0:=\left(\begin{array}{cc}
-h^2\Delta + V_1 & 0\\
0 & -h^2\Delta + V_2
\end{array}\right),
$$
with eigenvalue $\lambda_j$ imbedded in its continuous spectrum $[\Gamma, +\infty)$.

\begin{theorem}\sl
\label{mainth}
Suppose Assumptions 1-3, (\ref{gap2}), and Assumption [V] or [NT]. Let $g\in L^\infty (\R)$ supported in $(I(h)+(-2a(h), 2a(h)))$ with $g=1$ on $I(h)+[-a(h),a(h)]$, and such that, for some $\nu \geq 0$, one has,
\be
\begin{aligned}
& g, g', \dots, g^{(\nu)} \in L^\infty (R);\\
& g^{(k)} ={\mathcal O}(a(h)^{-k})\quad (k=1,\dots,\nu).
\end{aligned}
\ee
 Then, for all $t\in\R$ and $\varphi \in {\rm Span}\{ \phi_1,\dots,\phi_m\}$, one has,
\be
\label{mainres}
\la e^{-itH}g(H)\varphi, \varphi\ra =\sum_{j=1}^m e^{-it\rho_j}b_j(\varphi, h) + r(t,\varphi,h),
\ee
where $\rho_1,\dots,\rho_m$ are the resonances of $H$ lying in $\Omega(h):=I(h)+[-a(h),a(h)] -i[0,\varepsilon_1]$,  and satisfy,
\be
\label{estres}
\rho_j =\lambda_j +{\mathcal O}(h^2),
\ee
$r(t,\varphi,h)$ is such that,  
\be
\label{estrestth}
r(t,\varphi,h)={\mathcal O}\left(  \frac{h^2}{a(h)} \min_{0\leq k \leq \nu }\{ \frac{1}{[(1+t)a(h)]^{k}}\}\|\varphi\|^2\right),
\ee
uniformly with respect to $h>0$ small enough, $t\geq 0$, and $\varphi\in{\rm Span}(\phi_1,\dots\phi_m)$.
Here  $b_j(\varphi,h)$ is the residue at $\rho_j$ of the meromorphic extension from $\{\Im z>0\}$ of the function,
$$
z\mapsto \la (z-H)^{-1}\varphi, \varphi\ra.
$$
and satisfies: There exists a $m\times m$ matrix $M(z)$ depending analytically on $z\in \Omega(h)$, with
\be
\label{estM}
\| M(z)\| ={\mathcal O}(h^2),
\ee
such that,
\be
\label{aj=residu}
\begin{aligned}
& b_j(\varphi,h) \mbox{  is the residue at } \rho_j \mbox{ of the meromorphic function}\\
&z \mapsto \la (z-\Lambda +M(z))^{-1}\alpha_\varphi, \alpha_\varphi\ra_{\C^m},
\end{aligned}
\ee
where $\alpha_\varphi :=(\la \varphi, \phi_1\ra,\dots,\la \varphi, \phi_m\ra)$ and $\Lambda:={\rm diag}(\lambda_1,\dots,\lambda_m)$.

If in addition one assumes that $\lambda_1,\dots,\lambda_m$ are all simple, and the gap $\widetilde a(h):= \min_{j\not=k}|\lambda_j -\lambda_k|$ is such that,
\be
\label{gaplambda}
h^2/\widetilde a(h)\to 0\, \mbox{ as }\, h\to 0_+,
\ee
then, $b_j(\varphi,h)$ satisfies,
\be
\label{resutres}
b_j(\varphi ,h)=|\la \varphi, \phi_j\ra|^2 +{\mathcal O}\left( (h^2 + h^4(a\widetilde a)^{-1})\|\varphi\|^2\right),
\ee
uniformly with respect to $h>0$ small enough and $\varphi\in{\rm Span}(\phi_1,\dots\phi_m)$.
\end{theorem}
\begin{remark}\sl
Actually, our proof also gives a generalization of a result given in \cite{CGH} for the case $m=1$ : see Propositions \ref{genCGH} and \ref{genCGH1}.
\end{remark} 

\begin{remark}\sl 
\label{remr0}
Concerning the remainder term, we will  see  in the proof that it  is of the form 
$$
 r(t, \varphi,h)= r_0(t, \varphi,h) +{\mathcal O}\left(  \frac{h^4}{(a(h))^2} \min_{0\leq k \leq \nu }\{ \frac{1}{[(1+t)a(h)]^{k}}\}\|\varphi\|^2\right)
$$
with
\be
\label{r0lim}
\begin{aligned} r_0(t, \varphi,h) = & -h^2\lim_{\varepsilon, \varepsilon' \to 0_+ } \sum_{j,k}  \la \varphi,\phi_j \ra \overline{  \la \varphi, \phi_k\ra} \\ &\times \la e^{-itP_2} g(P_2) (P_2-\lambda_j-i\varepsilon)^{-1} W^* u_j, (P_2-\lambda_k + i\varepsilon')^{-1} W^* u_k \ra. \end {aligned}
\ee
\end{remark}
\begin{remark}\sl
\label{rembeyondlt}
In particular,  one has $\vert  r(t, \varphi,h)\vert  << \vert  e^{-it\rho_j}\vert $ as long as $ 0\leq t<< \frac {1}{ \vert \Im \rho_j\vert} \ln( a(h) / h^2)$ that is much beyond the life time. In the case $ \nu \geq1$, since $\vert \Im \rho_j\vert$ is exponentially small w.r.t. $h$ , this can indeed be improved by  allowing times up to $ \frac{C_0 \nu}{h\vert \Im \rho_j\vert}$ for some constant $C_0>0$ independent of $\nu$.
\end{remark}
\begin{remark}\sl Let us observe that, in the particular case where $m=1$, one obtains $b_1(\varphi ,h)=|\la \varphi, \phi_1\ra|^2 +{\mathcal O}\left( (h^2 + h^4/a^2)\|\varphi\|^2\right)$. Therefore, in the situation of the Theorem with (\ref{gaplambda}), a mere application of the previous result for each $\lambda_j$ would give $b_j(\varphi ,h)=|\la \varphi, \phi_j\ra|^2 +{\mathcal O}\left( (h^2 + h^4/\widetilde a^2)\|\varphi\|^2\right)$, and  compared with (\ref{resutres}) this is a weaker result if $\,\widetilde a(h) << a(h)$.
\end{remark}

As a corollary, for the case without energy cutoff, we also obtain,
\begin{corollary}\sl
\label{corollaire} In the general situation of Theorem \ref{mainth} (without the assumption on the simplicity of the $\lambda_j$'s), one has,
$$
\la e^{-itH}\varphi, \varphi\ra =\sum_{j=1}^m e^{-it\rho_j}b_j(\varphi, h)+{\mathcal O}\left( (h^2 + h^4a(h)^{-2})\|\varphi\|^2\right).
$$
\end{corollary}
In the sequels, we will concentrate on the detailed proof of Theorem \ref{mainth} in the case of Assumption [V]. The more general case of Assumption [NT] can be proved in a similar way by using the Helffer-Sj\"ostrand framework of resonances theory \cite{HeSj2}, and will be outlined in Section \ref{casNT}.
%%%%%%%%%%%%%%%%%%%%%%%
\section{Preliminaries}
\label{secprelim}
In order to prove Theorem \ref{mainth}, we start from the Stone formula,
\be
\label{stone}
\la e^{-itH}g(H)\varphi, \varphi\ra =\lim_{\varepsilon \to 0_+}\frac1{2i\pi}\int_\R e^{-it\lambda} g(\lambda)\la (R(\lambda+i\varepsilon)-R(\lambda-i\varepsilon))\varphi, \varphi\ra d\lambda,
\ee
where $R(z):=(H-z)^{-1}$. In the sequels, we also denote by $R_\theta (z):=(H_\theta -z)^{-1}$ the distorted resolvent, and  by $\varphi_\theta:=U_{i\theta}\varphi$ the distortion of $\varphi$ (observe that, thanks to the analyticity of $V_1$ and the ellipticity of $P_1$, each function $u_j$ can be distorted without problem). In particular, by standard arguments (see, e.g., \cite{ReSi, HeMa}), one has $\la R(z)\varphi, \varphi\ra =\la R_\theta (z)\varphi_\theta, \varphi_{-\theta}\ra$. From now on, we fix $\theta >0$ small enough and,
thanks to the fact that $g=1$ on $I(h)+[-a(h),a(h)]$, we can slightly deform the contour of integration in this region, and rewrite (\ref{stone}) as,
\be
\label{stone1}
\begin{aligned}
\la e^{-itH}g(H)\varphi, \varphi\ra =\frac1{2i\pi}\int_{\gamma_+} e^{-itz} & g(\Re z)  \la R_\theta (z)\varphi_\theta, \varphi_{-\theta}\ra dz\\ &-\frac1{2i\pi}\int_{\gamma_-} e^{-itz} g(\Re z)\la R_{-\theta}(z)\varphi_{-\theta}, \varphi_\theta\ra dz,
\end{aligned}
\ee
where the complex contour $\gamma_\pm$ can be parametrized by $\Re z$, coincides with $\R$ away from $I(h)+(-a(h),a(h))$, and is included in $\{\pm \Im z >0\}$ on $I(h)+(-a(h),a(h))$.

Here we anticipate by using (\ref{estres}) and, proceeding as in \cite{CGH}, we see that (\ref{stone1}) can be transformed into,
\be
\label{stone2}
\la e^{-itH}g(H)\varphi, \varphi\ra = \sum_{j=1}^m e^{-it\rho_j}b_j(\varphi,  h)+ r(t,\varphi, h),
\ee
where $b_j(\varphi,  h)$ is the residue at $\rho_j$ of the meromorphic function 
$$
z\mapsto - \la R_\theta (z)\varphi_\theta, \varphi_{-\theta}\ra,
$$
and $r(t,\varphi, h)$ is given by,
\be
\label{reste}
r(t,\varphi, h):= \frac1{2i\pi}\int_{\gamma_-} e^{-itz} g(\Re z) \left(\la R_\theta (z)\varphi_\theta, \varphi_{-\theta}\ra -\la R_{-\theta}(z)\varphi_{-\theta}, \varphi_\theta\ra\right) dz,
\ee
where $\gamma_-$ is chosen in such a way that it stays below the $\rho_j$'s.
Thus, the proof will consist in estimating both $b_j(\varphi,  h)$ and $r(t,\varphi, h)$.
%%%%%%%%%%%%%%%%%%%%%%%
\section{The Grushin problems}
\label{grushin}
From now on (up to Section {\ref{casNT}), we suppose Assumption [V].

In order to have good enough estimates on the resolvent, and in particular to compare it with that of $P_1$, for $z$ in $\Omega (h):= (I(h)+[-a(h),a(h)])+i[-\varepsilon_1,\varepsilon_1]$, we specify our choice of distorsion.  In (\ref{distorsion}), we take $F$ such that,
\be
\left\{\begin{array}{l}
f(x)=x \, \mbox{ in a neighborhood of the sea } \{ V_2\leq 0\};\\
F=0 \, \mbox{ in a neighborhood of the well } U=\{V_1\leq 0\}.
\end{array}\right.
\ee

With such a distorsion, it is well known (see, e.g., \cite{BCD}) that, under Assumption [V], the distorted operator $P_2^\theta$ satifies,
\be
\label{ellP2theta'}
\| (P_2^\theta -z)^{-1}\|_{{\mathcal L}(L^2(\R^n))} ={\mathcal O}(1),
\ee
uniformly with respect to $h>0$ small enough and $z\in \Omega (h)$.

We introduce the two following Grushin problems,
$$
{\mathcal G}(z):=\left(\begin{array}{cc}
H_\theta-z & L_-\\
L_+& 0
\end{array}\right)\, :\,{\mathcal D}_H \times \C^m \to{\mathcal H} \times \C^m,
$$
$$
{\mathcal G}_0(z):=\left(\begin{array}{cc}
H_0^\theta-z & L_-\\
L_+& 0
\end{array}\right)\, :\, {\mathcal D}_H\times \C^m \to {\mathcal H} \times \C^m,
$$
where $H_0^\theta$ stands for the distorted Hamiltonian obtained from $H_0$, and $L_\pm$ are defined as,
\be
\label{L-}
L_-(\alpha_1,\dots,\alpha_m):=\sum_{j=1}^m \alpha_j\phi_j^\theta;
\ee
\be
\label{L+}
L_+u:=L_-^*u=(\la u,\phi_1^{-\theta}\ra,\dots,\la u,\phi_m^{-\theta}\ra).
\ee
with $\phi_j^{\pm\theta}:=U_{\pm i\theta}\phi_j$.

It is elementary to check that ${\mathcal G}_0(z)$ is invertible, with inverse given by,
$$
{\mathcal G}_0(z)^{-1}=\left(\begin{array}{cc}
\widehat\Pi_\theta\widehat R_0^\theta (z)\widehat\Pi_\theta & L_-\\
L_+& z-\Lambda
\end{array}\right),
$$
where  $\Lambda={\rm diag}(\lambda_1,\dots,\lambda_m)$, $\widehat\Pi_\theta :=1-\Pi_\theta$ with $\Pi_\theta$ the spectral projection of $H_0^\theta$ associated with the eigenvalues $(\lambda_1,\dots,\lambda_m)$, that is,
$$
\Pi_\theta u:=\sum_{j=1}^m \la u, \phi_j^{-\theta}\ra\phi_j^\theta,
$$
and $\widehat R_0^\theta (z)$ is the reduced resolvent of $H_0^\theta$ i.e.  the inverse of the restriction of $H_0^\theta -z$ to the range of $\widehat\Pi_\theta$. 

In addition to (\ref{ellP2theta'}), we have, 

\begin{lemma}\sl
\label{estresredG}
$$
\|(\widehat P_1^{\pm\theta} -z)^{-1}\|_{{\mathcal L}(L^2(\R^n))} ={\mathcal O}(a(h)^{-1}),
$$
uniformly with respect to $h>0$ small enough and $z\in\Omega (h)$.
\end{lemma}
\begin{proof}
See Appendix 1.
\end{proof}

In order to prove that ${\mathcal G}(z)$ is invertible, too, and to compare its inverse with  ${\mathcal G}_0(z)^{-1}$, we compute the product,
$$
{\mathcal G}(z){\mathcal G}_0(z)^{-1}=:\left(\begin{array}{cc}
A_{11}& A_{12}\\
A_{21}& A_{22}
\end{array}\right).
$$
Using that $H_\theta =H_0^\theta + h{\mathcal W}_\theta$ (where ${\mathcal W}_\theta$ stands for the distorted operator obtained from ${\mathcal W}$), we find,
$$
\begin{aligned}
& A_{11} = 1+h{\mathcal W}_\theta\widehat R_0^\theta (z);\\
& A_{12} = h{\mathcal W}_\theta L_- ;\\
& A_{21} =0;\\
& A_{22} = I_{\C^m}.
\end{aligned}
$$
Then, we observe,
$$
\widehat\Pi_\theta =\left(\begin{array}{cc}
\widehat\Pi^\theta_1& 0\\
0 & 1
\end{array}\right),
$$
and,
$$
\widehat R_0^\theta (z)=\left(\begin{array}{cc}
\widehat\Pi^\theta_1\widehat R_1^\theta (z)\widehat\Pi^\theta_1& 0\\
0 & R_2^\theta (z)
\end{array}\right),
$$
where $R_2^\theta (z)$ is the resolvent of $P_2^\theta$ (the distorted operator obtained from $P_2$), and $\widehat R_1^\theta (z)$ is the reduced resolvent of $P_1^\theta$. Thus,
denoting by $W_\theta$ the distorted operator obtained from $W=w(x,hD_x)$, and $W^*_\theta$ that obtained from $W^*$, we find,
$$
h{\mathcal W}_\theta\widehat R_0^\theta (z)=\left(\begin{array}{cc}
0& hW_\theta R_2^\theta (z)\\
hW_\theta^*\widehat\Pi^\theta_1\widehat R_1^\theta (z)\widehat\Pi^\theta_1 & 0
\end{array}\right).
$$
Here we must be aware that this operator is not ${\mathcal O}(h)$, since $\widehat\Pi^\theta_1\widehat R_1^\theta (z)\widehat\Pi^\theta_1$ is  ${\mathcal O}(a(h)^{-1})$ only. However, the other off-diagonal operator $hW_\theta R_2^\theta (z)$ is ${\mathcal O}(h),$ and this is enough, for instance, to invert $1+h{\mathcal W}_\theta\widehat R_0^\theta (z)$ without problem. From now on, we set,
\be
\label{estQ1Q2}
Q_1(z):= W_\theta^*\widehat\Pi^\theta_1\widehat R_1^\theta (z)\widehat\Pi^\theta_1={\mathcal O}(a(h)^{-1})\;;\;Q_2(z):=W_\theta R_2^\theta (z)={\mathcal O}(1).
\ee
In particular, 
\be
\label{K}
K(z):= h^2Q_1(z)Q_2(z)={\mathcal O}(h^2/a(h)),
\ee
 and thus, by assumption (\ref{gap2}), the operator $1-K$ is invertible for $h>0$ small enough. Then, a straightforward computation shows that ${\mathcal G}(z){\mathcal G}_0(z)^{-1}$ is invertible, with inverse given by,
 $$
 {\mathcal F}(z):=\left(\begin{array}{cc}
B_1(z)& B_2(z)\\
0& I_{\C^m}
\end{array}\right),
 $$
 where,
 $$
 B_1(z):=\left(\begin{array}{cc}
1+h^2Q_2(1-K)^{-1}Q_1& -hQ_2(1-K)^{-1}\\
-h(1-K)^{-1}Q_1& (1-K)^{-1}
\end{array}\right),
 $$
 and
 \be
 \label{B2}
 B_2(z):=h^2\left(\begin{array}{cc}
0& Q_2(1-K)^{-1}\\
0& (1-K)^{-1}
\end{array}\right){\mathcal W}_\theta L_-.
\ee
 (Here, we have also used the fact that  the first component of ${\mathcal W}_\theta\phi_j$ is identically 0.)
 
 A similar computation shows that ${\mathcal G}_0(z)^{-1}{\mathcal G}(z)$ is invertible, too, and, as a consequence, so is ${\mathcal G}(z)$, with inverse,
\be
 {\mathcal G}(z)^{-1} ={\mathcal G}_0(z)^{-1}{\mathcal F}(z)=\left(\begin{array}{cc}
E(z)& E_+(z)\\
E_-(z)&E_{-+}(z)
\end{array}\right),
\ee
where,
\be
\label{G-1H0}
\begin{aligned}
& E(z):=\widehat\Pi_\theta\widehat R_0^\theta (z)\widehat\Pi_\theta B_1(z)\\
& E_+(z):=L_- +\widehat\Pi_\theta\widehat R_0^\theta (z) B_2(z) \\
& E_-(z):=L_+B_1\\
& E_{-+}(z):=z-\Lambda+L_+B_2(z).
\end{aligned}
\ee
We set,
$$
M(z):=L_+B_2(z)=M_0(z)+ M_1(z),
$$
with
$$
M_0(z):= h^2L_+\left(\begin{array}{cc}
0& Q_2(z)\\
0& 1
\end{array}\right){\mathcal W}_\theta L_- ,
$$
and
\be
\label{M1}
M_1(z):=L_+B_2(z)-M_0(z).
\ee
One can prove,
\begin{lemma}\sl 
\label{estM0M1}
One has,
$$
\begin{aligned}
& \| M_0(z)\|_{{\mathcal L}(\C^m)}={\mathcal O}(h^2);\\
& \| M_1(z)\|_{{\mathcal L}(\C^m)}={\mathcal O}(h^4/a(h))=o(h^2),
\end{aligned}
$$
uniformly with respect to $h>0$ small enough and $z\in\Omega (h)$.
\end{lemma}
\begin{proof}
See Appendix 2.
\end{proof}

In particular,
\be
\label{estM1}
\| M(z)\|_{{\mathcal L}(\C^m)}={\mathcal O}(h^2),
\ee
uniformly with respect to $z\in\Omega (h)$ and $h>0$ small enough. Since $h^2/a(h)\to 0$, by standard perturbation theory we deduce,
$$
{\rm Sp}(\Lambda + M(z))=\{\lambda_1(z),\dots,\lambda_m(z)\},
$$
with, 
$$
\lambda_j(z) =\lambda_j+{\mathcal O}(h^2).
$$
As a consequence the solutions $z\in\Omega (h)$ of the problem,
$$
0\in \sigma (E_{-+}(z)),
$$
are all of the form, 
$$
z=\lambda_j +{\mathcal O}(h^2),
$$
for some $j$. Deforming continuously $(E_{-+}(z))$ into $z-\Lambda$ (e.g., by setting $\Lambda_t(z):= z-\Lambda + tM(z)$, $0\leq t\leq 1$), and following continuously the roots of the determinant of $\Lambda_t(z)$ as $t$ varies from 0 to 1, we also see that all the values of $j$ are reached by such solutions.
Since we also know that these solutions are precisely the resonances of $H$ in $\Omega(h)$ (see (\ref{resolv})), we have proved (\ref{estres}).
%%%%%%%%%%%%%%%%%%%%%%%%%%

\section{The reduced resolvent}
In this section, we still consider the Grushin problem given by ${\mathcal G}(z)$, but we will solve it in a different way, in order to obtain the inverse in terms of the reduced resolvent $\widehat R_\theta (z)$ of $H_\theta$ (instead of that of $H_0^\theta$), as in the usual Feshbach method. 

Indeed, denoting by $\widehat H_\theta$ the restriction of $\widehat\Pi_\theta H_\theta$ to the range of $\widehat\Pi_\theta$, for all $z$ such that $\Im z>0$ we can define the reduced resolvent
$\widehat R_\theta (z)$ as the inverse of $\widehat H_\theta-z$, and it is straightforward to verify that, for such $z$, the inverse of ${\mathcal G}(z)$ is given by,
$$
 {\mathcal G}(z)^{-1} =\left(\begin{array}{cc}
E(z)& E_+(z)\\
E_-(z)&E_{-+}(z)
\end{array}\right),
$$
with,
\be
\label{resred1}
\begin{aligned}
& E(z):=\widehat\Pi_\theta\widehat R_\theta (z)\widehat\Pi_\theta \\
& E_+(z):=(1-h\widehat\Pi_\theta\widehat R_\theta (z)\widehat\Pi_\theta {\mathcal W}_\theta)L_-  \\
& E_-(z):=L_+(1-h{\mathcal W}_\theta \widehat\Pi_\theta\widehat R_\theta (z)\widehat\Pi_\theta)\\
& E_{-+}(z):=z-\Lambda+h^2(\la {\mathcal W}_\theta \widehat\Pi_\theta\widehat R_\theta (z)\widehat\Pi_\theta {\mathcal W}_\theta \phi_k^\theta, \phi_j^{-\theta}\ra)_{1\leq j,k\leq m}.
\end{aligned}
\ee
Comparing with (\ref{G-1H0}), we obtain in particular (still for $\Im z>0$, for which the computations of the previous section remain valid),
\be
\label{resred2}
\widehat\Pi_\theta\widehat R_\theta (z)\widehat\Pi_\theta=\widehat\Pi_\theta\widehat R_0^\theta (z)\widehat\Pi_\theta B_1(z).
\ee
Now, since both expressions are holomorphic in $\{\Im z>0\}$, and the right-hand side extends analytically in $\Omega (h)$, we conclude that so does  $\widehat\Pi_\theta\widehat R_\theta (z)\widehat\Pi_\theta$, and the identity remains valid in $\Omega (h)$. 

In addition, the expression $\la W_\theta \widehat\Pi_\theta\widehat R_\theta (z)\widehat\Pi_\theta W_\theta \phi_k^\theta, \phi_j^{-\theta}\ra$ is actually independent of $\theta$, and is nothing but the meromorphic extension to $\Omega (h)$ of the function (holomorphic in $\{\Im z >0\}$),
\be
\label{Fjk}
F_{j,k}(z):=\la {\mathcal W} \widehat\Pi\widehat R (z)\widehat\Pi {\mathcal W} \phi_k, \phi_j\ra
\ee
Finally, in order to estimates the residues appearing in (\ref{stone2}), let us recall the well known formula for the whole resolvent of $H_\theta$. For $z\in \Omega (h)\backslash \{\rho_1,\dots,\rho_m\}$, one has,
\be
\label{resolv}
R_\theta (z)=E(z) - E_+(z)\left( E_{-+}(z)\right)^{-1}E_-(z).
\ee
In view of (\ref{resred1})-(\ref{resred2}), we know that the operators $E(z)$, $E_\pm (z)$ and $E_{-+}(z)$ depend analytically on $z$ in $\Omega (h)$. Therefore, in formula (\ref{resolv}), the only possible poles come from $\left( E_{-+}(z)\right)^{-1}$. 

Therefore, we have proved,
\begin{proposition}\sl
\label{genCGH}
The distorted resolvent $R_\theta (z)$ of $H$ is given by (\ref{resolv}), where the operators $E(z)$, $E_\pm (z)$ and $E_{-+}(z)$ are given in (\ref{resred1}). Moreover, the resonances of $H$ in $\Omega (h)$ are exactly the roots of the equation,
$$
{\rm det} (z-\Lambda + h^2 F(z))=0,
$$
where $F(z)$ is the $m\times m$ matrix with coefficients $F_{j,k}(z)$ ($1\leq j,k\leq m$) given by (\ref{Fjk}).
\end{proposition}
In the particular case where $m=1$, let us observe that, at first glance, $F_{1,1}(z)$ can be estimated by ${\mathcal O}(a(h)^{-1})$, and its holomorphic derivative $F_{1,1}'(z)$ by ${\mathcal O}(a(h)^{-2})$ (this is because of the presence of the reduced resolvent in $F_{1,1}(z)$). For the resonance, this leads to, 
$$
\rho_1=\lambda_1-h^2F_{1,1}(\rho_1)=\lambda_1 +{\mathcal O}(h^2/a(h)) =\lambda_1 -h^2F_{1,1}(\lambda_1)+{\mathcal O}(h^4/a(h)^3),
$$
which, compared to the result given in \cite{CGH} seems much less interesting.
But actually, looking more precisely to the expression of $F(z)$, one can prove,
\begin{lemma}\sl In the case $m=1$, one has,
$$
|F_{1,1}(z)|+|F_{1,1}'(z)| ={\mathcal O}(1),
$$
uniformly with respect to $h>0$ small enough and $z\in\Omega (h)$.
\end{lemma}
\begin{proof} Using (\ref{resred2}), we have,
$$
\begin{aligned}
F_{1,1}(z) &= \la {\mathcal W}_{\theta} \widehat\Pi_\theta \widehat R_0^\theta (z)\widehat\Pi_\theta B_1(z) {\mathcal W}_\theta \phi_1^\theta, \phi_1^{-\theta}\ra\\
&= \la   B_1(z) {\mathcal W}_\theta \phi_1^\theta, \widehat\Pi_{-\theta} \widehat R_0^{-\theta} (\overline{z})\widehat\Pi_{-\theta}({\mathcal W}_{\theta})^*\phi_1^{-\theta}\ra,
\end{aligned}
$$
and since,
$$
({\mathcal W}_{\theta})^*\phi_1^{-\theta}=\left(\begin{array}{c}
0\\ W_{-\theta}u_j^{-\theta},
\end{array}\right)
$$
we have,
\be
\label{eq7.5}
\widehat\Pi_{-\theta} \widehat R_0^{-\theta} (\overline{z})\widehat\Pi_{-\theta}({\mathcal W}_{\theta})^*\phi_1^{-\theta}=\left(\begin{array}{c}
0\\ R_2^{-\theta}(\overline{z})W_{-\theta}u_j^{-\theta},
\end{array}\right)
\ee
Hence, $\|\widehat\Pi_{-\theta} \widehat R_0^{-\theta} (\overline{z})\widehat\Pi_{-\theta}({\mathcal W}_{\theta})^*\phi_1^{-\theta}\|_{L^2}={\mathcal O}(1)$, and since also
$\|B_1(z)\|_{{\mathcal L}(L^2)} ={\mathcal O}(1)$, we deduce,
$$
F_{1,1}(z)={\mathcal O}(1).
$$
(Here, we have used the fact that $\|u_j^{\pm\theta}\|_{L^2}={\mathcal O}(1)$.)

On the other hand, taking the derivate with respect to $z$, we obtain,
$$
F_{1,1}'(z)=\la {\mathcal W}_\theta \widehat\Pi_\theta\widehat R_\theta (z)^2\widehat\Pi_\theta {\mathcal W}_\theta \phi_1^\theta, \phi_1^{-\theta}\ra
$$
Then, applying (\ref{resred2}) with $\theta$ replaced by $-\theta$, and $z$ replaced by $\overline{z}$, and then taking the adjoint, we obtain,
\be
\label{resred2'}
\widehat\Pi_\theta\widehat R_\theta (z)\widehat\Pi_\theta=B^*_1(z)\widehat\Pi_\theta\widehat R_0^\theta (z)\widehat\Pi_\theta ,
\ee
with $B^*(z) =I+{\mathcal O}(h^2/a)$ in ${\mathcal L}(L^2)$. Using both (\ref{resred2}) and (\ref{resred2'}), we are led to,
$$
F_{1,1}'(z)=\la B_1^*(z) \widehat\Pi_\theta\widehat R_0^\theta (z)\widehat\Pi_\theta {\mathcal W}_\theta \phi_1^\theta,B_1(z)^* \widehat\Pi_{-\theta}\widehat R_0^{-\theta} (\overline{z})\widehat\Pi_{-\theta} ({\mathcal W}_\theta)^* \phi_1^{-\theta}\ra.
$$
Thus, we can conclude as before (see (\ref{eq7.5})) that $F_{1,1}'(z)={\mathcal O}(1)$.
\end{proof}

As a consequence, we obtained the following generalization of the result of \cite{CGH}:
\begin{theorem}\sl 
\label{genCGH1}
Suppose Assumptions 1-3,  (\ref{gap2}), and $m=1$. Then, the resonance $\rho_1(h)$ of $H$ that is the closest one to $\lambda_1(h)$ satisfies,
$$
\rho_1(h) =\lambda_1(h) -h^2F_{1,1}(\lambda_1(h)) + {\mathcal O}(h^4),
$$
uniformly for $h>0$ small enough. Here, $F_{1,1}(z)$ is defined in (\ref{Fjk}).
\end{theorem}
%%%%%%%%%%%%%%%%%%%%%%%%%
\section{Estimates on the residues}
\label{secResidu}

Going back to (\ref{stone2}), and using (\ref{resolv}), we deduce,
\be
\label{resid1}
b_j(\varphi,h)={\rm Residue}_{z=\rho_j}\la E_+(z)\left( E_{-+}(z)\right)^{-1}E_-(z)\varphi_\theta, \varphi_{-\theta}\ra.
\ee
Since $\varphi \in {\rm Span}(\phi_1,\dots,\phi_m)$, it can be written as,
\be
\varphi =\sum_{j=1}^m \alpha_j\phi_j,
\ee
($\alpha_j\in \C$), and thus we see on (\ref{resred1}) that we actually have,
$$
E_-(z)\varphi_\theta = L_+\varphi_\theta = (\alpha_1,\dots,\alpha_m).
$$
In a similar way, since $\widehat\Pi_\theta^* =\widehat\Pi_{-\theta}$, we also find,
$$
E_+(z)^*\varphi_{-\theta}= (\alpha_1,\dots,\alpha_m).
$$
Inserting into (\ref{resid1}), and setting,
$$
\alpha_\varphi:= (\alpha_1,\dots,\alpha_m)\in \C^m,
$$
we obtain,
\be
\label{resid2}
b_j(\varphi,h)={\rm Residue}_{z=\rho_j}\la  E_{-+}(z)^{-1}\alpha_\varphi, \alpha_\varphi\ra_{\C^m}.
\ee
Therefore, using (\ref{G-1H0})-(\ref{estM1}), we deduce (\ref{estM})-(\ref{aj=residu}).

Now, assuming that the $\lambda_j$'s are simple and that (\ref{gaplambda}) is satisfied,
we write,
\be
\label{M0M1'}
E_{-+}(z) = (z-\Lambda+M_0(z))\left ( 1 +(z-\Lambda+M_0(z))^{-1}M_1(z)\right).
\ee
Moreover, using (\ref{resid2}) and denoting by $\gamma_j$ the oriented boundary of the disc centered in $\lambda_j$ of radius $\widetilde a(h)/2$, we have,
\be
\label{intcont}
b_j(\varphi, h)=\frac1{2i\pi}\int_{\gamma_j} \la  E_{-+}(z)^{-1}\alpha_\varphi, \alpha_\varphi\ra dz.
\ee

When $z\in\gamma_j$,  we have $\|(z-\Lambda)^{-1}\| ={\mathcal O}(\widetilde a^{-1})$ and thus, using (\ref{gaplambda}),
$$
 \| (z-\Lambda+M_0(z))^{-1}\| =\| (1+(z-\Lambda)^{-1}M_0(z))^{-1}(z-\Lambda)^{-1}\| ={\mathcal O}(\widetilde a^{-1}).
$$
Moreover, using (\ref{estM1}), we have,
$$
\| (z-\Lambda+M_0(z))^{-1}M_1(z)\| ={\mathcal O}(h^4/(a\widetilde a))=o(1),
$$
and thus,  by (\ref{M0M1'}), for $z\in \gamma_j$,
$$
\begin{aligned}
E_{-+}(z)^{-1}&=\left(1+{\mathcal O}(h^4/(a\widetilde a)\right)(z-\Lambda - M_0(z))^{-1}\\
&=(z-\Lambda +M_0(z))^{-1}+{\mathcal O}(h^4/(a\widetilde a^2)),
\end{aligned}
$$
and thus, since the length of $\gamma_j$ is ${\mathcal O}( \widetilde a)$,
\be
\label{EST1'}
\int_{\gamma_j} \la  E_{-+}(z)^{-1}\alpha_\varphi, \alpha_\varphi\ra dz =\int_{\gamma_j}\la (z-\Lambda +M_0(z))^{-1}\alpha_\varphi, \alpha_\varphi\ra dz +{\mathcal O}\left(\frac{h^4}{a\widetilde a}\right)\|\varphi\|^2.
\ee

On the other hand, we see on its definition that we have,
$$
M_0(z) = h^2\left( \la W_\theta R_2^\theta(z) W_\theta^* u_k^\theta , u_j^{-\theta}\ra \right)_{1\leq j,k\leq m},
$$
and, introducing the operator $\widetilde P_2:=-h^2\Delta + \widetilde V_2$ where $\widetilde V_2$ is as in (\ref{puitsbouche}), the exponential decay of $u_j^{\pm\theta}$ away from $U$ and Agmon estimates (see \cite{HeSj2}) show that,
$$
\la W_\theta R_2^\theta W_\theta^* u_k^\theta , u_j^{-\theta}\ra =\la W \widetilde R_2(z) W^* u_k , u_j\ra +{\mathcal O}(e^{-\delta /h}),
$$
for some constant $\delta >0$, and with $\widetilde R_2(z):=(\widetilde P_2 -z)^{-1}$. Setting
$$
\widetilde M_0(z):=\left( \la W_\theta R_2^\theta W_\theta^* u_k^\theta , u_j^{-\theta}\ra =\la W \widetilde R_2(z) W^* u_k , u_j\ra\right)_{1\leq j,k\leq m},
$$
we deduce as before,
\be
\label{EST2}
\int_{\gamma_j} \la  E_{-+}(z)^{-1}\alpha_\varphi, \alpha_\varphi\ra dz =\int_{\gamma_j}\la (z-\Lambda + \widetilde M_0(z))^{-1}\alpha_\varphi, \alpha_\varphi\ra dz +{\mathcal O}\left(\frac{h^4}{a\widetilde a}\right)\|\varphi\|^2,
\ee
where the matrix $\widetilde M_0(z)$ is ${\mathcal O}(h^2)$, depends analytically on $z\in\Omega (h)$, and is selfadjoint when $z$ is real. As a consequence, thanks to the gap condition on the $\lambda_j$'s, we see that the matrix $\Lambda -\widetilde M_0(z)$ can be diagonalized in a basis $(e_1(z), \dots, e_m(z))$ of $\C^m$, that depends analytically on $z\in \Omega (h)$, is orthonormal when $z$ is real, and the corresponding change of basis is given by a matrix $A(z)$ satisfying,
$$
\begin{aligned}
& {}^t A(z) A(z) =I_{\C^m};\\
& A(z)=I_{\C^m}+{\mathcal O}(h^2);\\
& {}^t A(z)(z-\Lambda +\widetilde M_0(z))^{-1}A(z) ={\rm diag}\left( \frac1{z-\mu_1(z)},\dots,\frac1{z-\mu_m(z)}\right),
\end{aligned}
$$
where the eigenvalues $\mu_1(z),\dots,\mu_m(z)$ of $\Lambda -\widetilde M_0(z)$ satisfy,$$
\mu_j (z) =\lambda_j + f_j(z)
$$
with $f_j(z)={\mathcal O}(h^2)$. Note that  $f_j$ are  real on the real. Since
$$
\frac{d}{dz} \widetilde M_0(z)={\mathcal O}(h^2),
$$
we see by a standard Hellmann-Feynman argument that, in this situation, we also have,
$$
\mu_j'(z)=f_j'(z) ={\mathcal O}(h^2).
$$
Moreover, the poles $\widetilde \lambda_1,\dots,\widetilde \lambda_m$ of $\la (z-\Lambda - \widetilde M_0(z))^{-1}\alpha_\varphi, \alpha_\varphi\ra $ are the solutions of an equation,
$$
z=\mu_j(z)
$$
for some $j=1,\dots,m$. Thus, they are necessarily simple, and since $\mu_j(\overline{z})=\overline{\mu_j(z)}$, they must  be real.
Finally, we obtain,
\be
\label{EST3'}
\begin{aligned}
\frac1{2i\pi}\int_{\gamma_j}\la (z-\Lambda - \widetilde M_0(z))^{-1}\alpha_\varphi, \alpha_\varphi\ra dz &= (1-f_j'(\widetilde\lambda_j))^{-1}|\alpha_j|^2+{\mathcal O}(h^2\|\varphi\|^2)\\ &=|\alpha_j|^2+{\mathcal O}(h^2\|\varphi\|^2),
\end{aligned}
\ee
and (\ref{resutres}) follows from (\ref{intcont}), (\ref{EST2}) and (\ref{EST3'}).

\section{Estimates on the rest}
\label{secRest}

We have to estimate the quantity,
\be
S_{\theta}(z):=\la R_{\theta} (z)\varphi_{\theta}, \varphi_{-\theta}\ra -\la R_{-\theta}(z)\varphi_{-\theta}, \varphi_{\theta}\ra,
\ee
for $z\in \gamma_-$ where, setting $\widetilde I=[\alpha,\beta]:= I(h)+[-a,a]$, we choose the contour $\gamma_-$ as,
$$
\gamma_-:= (\R\backslash \widetilde I)\cup (\alpha-i[0, a])\cup ([\alpha,\beta]-ia)\cup (\beta-i[0,a]).
$$
We first compute $v=(v_1,v_2):=R_{\theta} (z)\varphi_{\theta}$. Denoting by $u:=\sum_j \alpha_j u_j$ the first component of $\varphi$, we find,
$$
\begin{aligned}
& v_1= (P_1^\theta -z)^{-1}(1-T_\theta)^{-1}u_\theta \\
& v_2= -h(P_2^\theta -z)^{-1}W_\theta^*v_1,
\end{aligned}
$$
with
$$
T_\theta:=h^2W_\theta(P_2^\theta -z)^{-1}W_\theta^*(P_1^\theta -z)^{-1}={\mathcal O}(h^2/a(h)).
$$
Then, using that $(P_1 -z)^{-1}u = \sum_k \alpha_k(\lambda_k-z)^{-1}u_k$, that the $u_j$'s are orthogonal to each other, and that $z$ stays at a distance greater than $a/2$ from the $\lambda_j$'s, we deduce,
$$
\la R_{\theta} (z)\varphi_{\theta}, \varphi_{-\theta}\ra=\la v_1, u_{-\theta}\ra=\sum_{j}\frac{\vert\alpha_j\vert^2}{\lambda_j-z}
+\sum_{j,k}\frac{\alpha_j\overline{\alpha_k}}{\lambda_k-z}\la T_\theta u_j^\theta, u_k^{-\theta}\ra +{\mathcal O}(h^4a^{-3}).
$$
Using again that the $u_j$'s are eigenfunction of $P_1$, this lead us to,
$$
\begin{aligned}
\la R_{\theta} (z)\varphi_{\theta}, \varphi_{-\theta}\ra=\sum_{j}\frac{\vert\alpha_j\vert^2}{\lambda_j-z}
+h^2\sum_{j,k}\frac{\alpha_j\overline{\alpha_k}}{(\lambda_k-z)(\lambda_j-z)}&\la W_\theta R_2^\theta(z) W_\theta^*u_j^\theta, u_k^{-\theta}\ra\\
&  +{\mathcal O}(h^4a^{-3}\Vert\varphi\Vert^2).
\end{aligned}
$$
Here we observe that the quantity $\la W_\theta R_2^\theta(z) W_\theta^*u_j^\theta, u_k^{-\theta}\ra$ is nothing but the holomorphic continuation from $\{\Im z >0\}$ through the real axis of the function $z\mapsto \la  R_2(z) W^*u_j, W^*u_k\ra$. From now on, we denote this continuation by $\la  \widetilde R_2(z) W^*u_j, W^*u_k\ra$.

Changing $\theta$ into $-\theta$, we also find an analog expression for $\la R_{-\theta} (z)\varphi_{-\theta}, \varphi_{\theta}\ra$, and making their difference, we obtain,
\be
\label{estStheta}
S_\theta (z)=h^2\sum_{j,k}\frac{\alpha_j\overline{\alpha_k}}{(\lambda_k-z)(\lambda_j-z)}\la  (\widetilde R_2(z)-R_2(z)) W^*u_j, W^*u_k\ra +{\mathcal O}(h^4a^{-3}\Vert\varphi\Vert^2).
\ee
Multiplying by $e^{-itz}g(\Re z)$ and integrating over $\gamma_-$, we obtain the required estimate of $r(t,\varphi, h)$ in the case $\nu =0$

For the case $\nu >0$, as in \cite{CGH} we use the formula,
$$
e^{-izt} =(1+t)^{-\nu}\left( 1+i\frac{d}{dz}\right)^\nu e^{-izt},
$$
and we make $k$ integrations by parts with respect to $z$ ($0\leq k\leq \nu$). This makes appear the composition of a finite number of resolvents and additional negative powers of $\lambda_j-z$, and the estimate follows in the same way.

Moreover, setting
\be
\label{r0}
r_0(t,\varphi, h):=\frac{h^2}{2i\pi}\sum_{j,k}\int_{\gamma_-}\frac{e^{-itz}g(\Re z)}{(\lambda_k-z)(\lambda_j-z)}\la  (\widetilde R_2(z)-R_2(z)) W^*u_j, W^*u_k\ra
\ee
we see on \eqref{estStheta} that we have $r(t,\varphi, h)=r_0(t,\varphi, h)+{\mathcal O}(h^4a^{-2}\Vert\varphi\Vert^2)$. In addition, in \eqref{r0}, we can change $(\lambda_j, \lambda_k)$ into $(\lambda_j+i\varepsilon, \lambda_k+i\varepsilon')$ and take the limit as $\varepsilon, \varepsilon'\to 0_+$. Before taking this limit, we can also deform $\gamma_-$ into $\R$, and this transforms $\widetilde R_2(z)-R_2(z)$ into $\widetilde R_2(\lambda+i0)-R_2(\lambda-i0) $. By the spectral theorem, this leads to the expression \eqref{r0lim} of Remark \ref{remr0}. \qed

\section{Proof of Corollary \ref{corollaire}}
\label{secCor}
We first prove,
\begin{lemma}\sl
$$
\sum_{j=1}^m b_j(\varphi, h)=\left(1+{\mathcal O}(h^2 +h^4/a^2)\right)\|\varphi\|^2.
$$
\end{lemma}
\begin{proof}
We write,
\be
\label{M0M1}
E_{-+}(z) = (z-\Lambda+M_0(z))\left ( 1 +(z-\Lambda+M_0(z))^{-1}M_1(z)\right),
\ee
and, using (\ref{resid2}) and denoting by $\gamma$ the oriented boundary of the rectangle $\{z\in\C\, ;\, \Re z \in I(h)+[-a(h), a(h)],\, |\Im z|\leq \varepsilon_1\}$, we have,
\be
\label{intcont'}
\sum_{j=1}^m b_j(\varphi, h)=\frac1{2i\pi}\int_\gamma \la  E_{-+}(z)^{-1}\alpha_\varphi, \alpha_\varphi\ra dz.
\ee
We divide $\gamma$ into its vertical part $\gamma^{\mathbf v}$ and its horizontal one $\gamma^{\mathbf h}$.

When $z\in\gamma^{\mathbf h}$, since $z$ remains at a distance $\varepsilon_1$ of $\R$, we have $\|(z-\Lambda)^{-1}\| ={\mathcal O}(1)$ and thus,
$$
 \| (z-\Lambda+M_0(z))^{-1}\| =\| (1+(z-\Lambda)^{-1}M_0(z))^{-1}(z-\Lambda)^{-1}\| ={\mathcal O}(1).
$$
Moreover, still for $z\in\gamma^{\mathbf h}$, we see on (\ref{K}) that $K(z)={\mathcal O}(h^2)$, and thus, by (\ref{B2}) and (\ref{M1}), $\|M_1(z)\|={\mathcal O}(h^4)$. As a consequence
$$
\| (z-\Lambda+M_0(z))^{-1}M_1(z)\| ={\mathcal O}(h^4), \quad (z\in \gamma^{\mathbf h}).
$$
Therefore, by (\ref{M0M1}), for such $z$ we can write,
$$
\begin{aligned}
E_{-+}(z)^{-1}&=\left(1+{\mathcal O}(h^4)\right)(z-\Lambda - M_0(z))^{-1}\\
&=(z-\Lambda - M_0(z))^{-1}+{\mathcal O}(h^4),
\end{aligned}
$$
and thus,
\be
\label{EST1}
\int_{\gamma^{\mathbf h}} \la  E_{-+}(z)^{-1}\alpha_\varphi, \alpha_\varphi\ra dz =\int_{\gamma^{\mathbf h}}\la (z-\Lambda - M_0(z))^{-1}\alpha_\varphi, \alpha_\varphi\ra dz +{\mathcal O}(h^4)\|\varphi\|^2.
\ee

On the other hand, when $z\in\gamma^{\mathbf v}$, we can write $z=z_1 +iz_2$ with $z_1,z_2\in\R$, ${\rm dist}(z_1, I(h))=a(h)$, $|z_2|\leq \varepsilon_1$. Therefore, for such $z$ we have,
$\| (z-\Lambda+M_0(z)))^{-1}\|  ={\mathcal O}((a+|z_2|)^{-1})$, $\|K(z)\|={\mathcal O}(h^2(a+|z_2|)^{-1})$, and $\| M_1(z)\|={\mathcal O}(h^4(a+|z_2|)^{-1})$. Proceeding as before, we deduce,
$$
E_{-+}(z)^{-1}=(z-\Lambda - M_0(z))^{-1}+{\mathcal O}(h^4/(a+|z_2|)^3)\|\varphi\|^2,
$$
and thus, integrating in $z_2$ on $[-\varepsilon_1,\varepsilon_1]$, 
\be
\label{EST2'}
\int_{\gamma^{\mathbf v}} \la  E_{-+}(z)^{-1}\alpha_\varphi, \alpha_\varphi\ra dz =\int_{\gamma^{\mathbf v}}\la (z-\Lambda - M_0(z))^{-1}\alpha_\varphi, \alpha_\varphi\ra dz +{\mathcal O}(h^4/a(h)^2)\|\varphi\|^2.
\ee
We deduce from (\ref{EST1})-(\ref{EST2'}),
\be
\label{EST3}
\int_{\gamma} \la  E_{-+}(z)^{-1}\alpha_\varphi, \alpha_\varphi\ra dz =\int_{\gamma}\la (z-\Lambda -M_0(z))^{-1}\alpha_\varphi, \alpha_\varphi\ra dz +{\mathcal O}(h^4/a(h)^2)\|\varphi\|^2.
\ee
At this point, we make the key observation that, by definition, $M_0(z)$ extends analytically in some $h$-independent complex neighborhood of $I(h)$, where it is ${\mathcal O}(h^2)$ in norm. As a consequence, modifying the complex contour $\gamma$ into another one that stays at some fix positive distance from $I(h)$, we deduce from (\ref{EST3}), 
\be
\label{ESTfin}
\int_{\gamma} \la  E_{-+}(z)^{-1}\alpha_\varphi, \alpha_\varphi\ra dz =\int_{\gamma}\la (z-\Lambda )^{-1}\alpha_\varphi, \alpha_\varphi\ra dz +{\mathcal O}(h^2 +h^4/a(h)^2)\|\varphi\|^2.
\ee
Going back to (\ref{intcont'}), this gives,
$$
\sum_{j=1}^m b_j(\varphi, h)=\sum_{j=1}^m |\alpha_j|^2+{\mathcal O}(h^2 +h^4/a(h)^2)\|\varphi\|^2 =\left(1+{\mathcal O}(h^2+h^4/a^2)\right)\|\varphi\|^2,
$$
and (\ref{resutres}) is proved.
\end{proof}

Now, applying Theorem \ref{mainth} with $t=0$, we obtain,
$$
\la g(H)\varphi, \varphi\ra = \sum_{j=1}^m b_j(\varphi, h) +{\mathcal O}(e^{-M/h})
$$
and thus, by the previous lemma,
$$
\la g(H)\varphi, \varphi\ra = \|\varphi\|^2 +{\mathcal O}(h^2+h^4/a^2)\|\varphi\|^2.
$$
Hence,
\be
\label{1-g}
\la (1-g(H))\varphi, \varphi\ra ={\mathcal O}(h^2+h^4/a^2)\|\varphi\|^2,
\ee
and we can chose $g$ in such a way that $0\leq g\leq 1$. In that case, (\ref{1-g}) can be re-written as,
$$
\| (1-g(H))^{\frac12}\varphi \|^2={\mathcal O}(h^2+h^4/a^2)\|\varphi\|^2,
$$
and Corollary \ref{corollaire} follows by writing,
$$
\begin{aligned}
\la e^{-itH}\varphi, \varphi\ra &=\la e^{-itH}g(H)\varphi, \varphi\ra+\la e^{-itH}(1-g(H))\varphi, \varphi\ra \\
&=\la e^{-itH}g(H)\varphi, \varphi\ra+\la e^{-itH}(1-g(H))^{\frac12}\varphi, (1-g(H))^{\frac12}\varphi\ra \\
&=\la e^{-itH}g(H)\varphi, \varphi\ra+{\mathcal O}(\| (1-g(H))^{\frac12}\varphi \|^2).
\end{aligned}
$$

\section{The non-trapping case}\label{casNT}

In the case when only Assumption [NT] is assumed (instead of Assumption [V]), the strategy of the proof is the same. However, an important ingredient for the estimates on the residues was the uniform boundedness of the resolvent of $P_2^\theta$. Therefore, in order to generalize this proof
one  needs a framework where  $(P_2-z)^{-1}$ becomes bounded uniformly with respect to $h$. This is provided by the theory of resonances developed by Helffer and Sj\"ostrand in \cite{HeSj2}. Without entering too much into details, let us just recall that this theory consists in changing $L^2(\R^n)$ into a space ${\mathcal H}_{\theta G}$, that contains $C_0^\infty (\R^n)$, and that depends on a positive small enough parameter $\theta$ and a function $G\in C^\infty (\R^{2n} ;\R)$ supported near $\{ p_2=0\}$ (where $p_2(x,\xi):=\xi^2 + V_2(x)$), and satisfying,
$$
| p_2 (x,\xi) -i\theta H_{p_2}G(x,\xi)| \geq \frac{\theta}{C}\la\xi\ra^2,
$$
for some constant $C>0$. Then, one has,
\be
\label{ellP2theta}
\| (P_2 -z)^{-1}\|_{{\mathcal L}({\mathcal H}_{\theta G})} ={\mathcal O}(1/\theta),
\ee
uniformly with respect to $h>0$ small enough and $z$ close to 0.
Let us also recall that pseudodifferential operators with analytic symbols on complex sectors can act on ${\mathcal H}_{\theta G}$, and their representation involves the restriction of their symbol to the complex Lagrangian manifold,
$$
\Lambda_{\theta G}:=\{ (x+i\theta \partial_\xi G(x,\xi), \xi -i\theta\partial_xG(x,\xi))\,;\, (x,\xi)\in \R^{2n}\}.
$$
Moreover, a whole symbolic calculus can be performed for such operators, where only the restrictions to $\Lambda_{\theta G}$ of the symbols are involved. Finally, as in the $L^2$-case, an analog of Sobolev spaces can be introduced by inserting a weight, and we denote by ${\mathcal H}_{\theta G}^2$ the analog of $H^2(\R^n)$ in this context. In particular, we have,
$$
P_1\, ,\, P_2\, : {\mathcal H}_{\theta G}^2\to {\mathcal H}_{\theta G}.
$$ 
Then, setting $D_{\theta G}:= {\mathcal H}_{\theta G}^2\times {\mathcal H}_{\theta G}^2$ and $\widetilde{\mathcal H}_{\theta G}:= {\mathcal H}_{\theta G}\times {\mathcal H}_{\theta G}$, 
we consider the two Grushin problems ${\mathcal G}(z)$ and ${\mathcal G}_0(z)$ as in Section \ref{grushin}, but this time without distortion, as operators : $D_{\theta G} \times \C^m \to\widetilde{\mathcal H}_{\theta G} \times \C^m$, and with the scalar product replaced (in the definition of $L_+$) by the duality-bracket between  $\widetilde{\mathcal H}_{\theta G}$ and $\widetilde{\mathcal H}_{-\theta G}$.

Then the proof  of the estimates on the residues proceeds in the same way, in particular  the fact that $G$ is supported near $\{p_2=0\}$ (thus, away from the well $U$) makes valid  an analog of Lemma \ref{estresredG} in this context. Indeed, the norm in ${\mathcal H}_{\theta G}$ is equivalent to a weighted norm of the same type as in (\ref{normL2G}), but this time with a weight $G$ that is no more compactly supported (but still supported in a neighborhood of $\{ p_2 =0\}$): see \cite{HeSj2}, Formula (9.48). For the same reason, the estimates of Lemma \ref{estM0M1} on $M_0(z)$ and $M_1(z)$ can be generalized, too, and all of Sections \ref{secResidu} and \ref{secCor} remain valid.

The same procedure applies to  estimate the remainder term $r(t,\varphi,h)$.
\section{Examples}
\label{secexample}

\subsection{The one dimensional case}
When $n=1$, if we assume,
$$
V_1'\not=0 \mbox{ on }\, \{ V_1=0\},
$$
then it is well known (see, e.g., {HeRo}) that the eigenvalues of $P_1$ are all simple and separated by a gap of order $h$. Then,
we can take $|I(h)|={\mathcal O}(h)$, $a=\widetilde a \sim h$ , and we also have $m={\mathcal O}(1)$. Moreover, in this case Assumption [NT] on $V_2$ is equivalent to,
$$
\begin{aligned}
& V_2'\not=0 \mbox{ on }\, \{ V_2=0\},\, \mbox{ and } \{V_2\leq 0\} \mbox{ has no }\\
& \mbox{bounded connected component}.
\end{aligned}
$$
For instance $V_2(x) =-\Gamma+\alpha (1+x^2)^{-1}$, (with $\alpha >0$  sufficiently large, so that $V_2>0$ on $\{V_1\leq 0\}$) satisfies all the assumptions (including Assumption [V]).

In such a situation, (\ref{resutres}) becomes,
\be
\label{a=h}
b_j(\varphi ,h)=|\la \varphi, \phi_j\ra|^2 +{\mathcal O}( h^2)\|\varphi\|^2,
\ee
and, with Corollary \ref{corollaire}, this gives,
\be
\label{atilde=h}
\la e^{-itH}\varphi, \varphi\ra =\sum_{j=1}^m e^{-it\rho_j}|\la \varphi, \phi_j\ra|^2+{\mathcal O}( h^2 )\|\varphi\|^2.
\ee

\subsection{The non-degenerate point-well}
In addition to Assumption 1, let us suppose,
$$
U=\{0\},\, {\rm Hess}\, V_1 (0) >0.
$$
Then, it is well known (see \cite{HeSj1, Si}) that the spectrum of $P_1$ near 0 consists of eigenvalues admitting asymptotic expansions as $h\to 0_+$, of the form,
$$
\lambda_j(h) \sim \sum_{k\geq 0} \lambda_{j,k}h^{1+\frac{k}2},
$$
where $\lambda_{j,0}$ is the $j$-th eigenvalue of the harmonic oscillator $-\Delta +\frac12\la  {\rm Hess}\, V_1 (0) x,x\ra$.

As for $V_2$, one can take $V_2(x) =-\Gamma+\alpha (1+x^2)^{-1}$ with $\alpha, \Gamma >0$ arbitrary. Then Assumption [V] is satisfied, and choosing $I(h)=[0, Ch]$ with $C\notin \{\lambda_{j,0}\, ;\, j\geq 1\}$, we see that the general assumptions of Theorem \ref{mainth} are satisfied with $a(h)\sim h$. Thus, (\ref{a=h}) remains valid in this case. 

Moreover, in the case $n=1$, all the $\lambda_{j,0}$'s are simple, and thus so are the $\lambda_j$'s, with a gap $\widetilde a \sim h$, and (\ref{atilde=h}) is valid, too.

When $n\geq 2$, some $\lambda_{j,0}$ may have some multiplicity. This is for instance the case if we take $n=2$ and $V_1(x_1,x_2) = x_1^2 +4x_2^2 +x_1^2x_2 + {\mathcal O}(|x|^4)$ uniformly near 0. Then (see \cite{HeSj1}, end of Section 3), the asymptotic of the first eigenvalues of $P_1$ can be computed, and one finds,
$$
\begin{aligned}
&\lambda_1(h) = 3h +{\mathcal O}(h^2);\\
&\lambda_2(h) = 5h +{\mathcal O}(h^{\frac32});\\
&\lambda_3(h) = 7h -\alpha h^{\frac32} +{\mathcal O}(h^2);\\
&\lambda_4(h) = 7h+\alpha h^{\frac32}  +{\mathcal O}(h^2);\\
&\lambda_5(h) = 9h +{\mathcal O}(h^{\frac32}),
\end{aligned}
$$
with $\alpha:=\int y_1^2y_2 v_1(y_1)w_2(y_2)v_3(y_1)w_1(y_2) dy_1dy_2 >0$, where $v_j$ stands for the normalized $j$-th eigenfunction of $-d_{y_1}^2 + y_1^2$, and $w_j$  for the normalized $j$-th eigenfunction of $-d_{y_2}^2 + 4y_2^2$.

Thus, we can apply Theorem \ref{mainth} with $I(h)=[0, 8h]$, $a(h)=h/2$, and $\widetilde a(h) = 2\alpha h^{\frac32}$.

\section{Appendix}
\subsection{Appendix 1: Proof of Lemma \ref{estresredG}}

We do it for $P_1^\theta$ only, since the sign of $\theta$ is not involved in the proof. Let $\eta, \psi, \chi \in C_0^\infty (\R^n)$ be such that,
$$
\begin{aligned}
& \inf_{\R^n} (V_1+\eta ) >0;\\
&\psi =1 \mbox{ in a neighborhood of } {\rm Supp }\,\eta;\\
& \chi =1 \mbox{ in a neighborhood of } {\rm Supp }\,\psi;\\
& {\rm Supp }\,\chi\,\subset \,\R^n\backslash {\rm Supp}\, F .
\end{aligned}
$$
We denote by,
$$
\widetilde P_1^\theta := P_1^\theta + \eta
$$
the perturbation of $P_1^\theta$ where the well $U$ has been filled with $\eta$ (the so-called ``filled-well'' operator). By analogy with a technique used in \cite{HeSj2}, Section 9 (in particular Formula (9.22)), we consider the operator,
$$
X(z):= \chi (\widehat P_1^\theta -z)^{-1}\psi + (\widetilde P_1^\theta -z)^{-1}(1-\psi).
$$
By a straightforward computation, we have,
\be
\label{approxresred}
(P_1^\theta -z)\widehat\Pi_1^\theta X(z)\widehat\Pi_1^\theta =\widehat\Pi_1^\theta +Y(z),
\ee
with,
$$
Y(z):=\widehat\Pi_1^\theta \left(-\chi \Pi_1^\theta \psi+[P_1^\theta,\chi](\widehat P_1^\theta -z)^{-1}\psi -\eta (\widetilde P_1^\theta -z)^{-1}(1-\psi )\right)\widehat\Pi_1^\theta.
$$
Then, denoting by $d_1$ the Agmon distance associated with $V_1$, one observes that both $d_1({\rm Supp}\,\nabla\chi, {\rm Supp}\,\psi)$ and $d_1( {\rm Supp}\,\eta, {\rm Supp}\,(1-\psi)$ are positive numbers. Therefore, one can apply e.g. the Propositions 9.3 and 9.4 in \cite{HeSj2} (or, more directly, Agmon estimates on $P_1^\theta$, uniformly with respect to $\theta$) to deduce the existence of some $\delta_1 >0$, independent of $\theta$, such that,
\be
\label{estA1}
\| [P_1^\theta,\chi](\widehat P_1^\theta -z)^{-1}\psi -\eta (\widetilde P_1^\theta -z)^{-1}(1-\psi )\|_{{\mathcal L}(L^2)}={\mathcal O}(e^{-2\delta_1/h}).
\ee
Moreover, since $\widehat\Pi_1^\theta \Pi_1^\theta =\Pi_1^\theta \widehat\Pi_1^\theta  =0$, we have,
$$
\widehat\Pi_1^\theta \left(\chi \Pi_1^\theta \psi)\right)\widehat\Pi_1^\theta =\widehat\Pi_1^\theta \left((\chi-1) \Pi_1^\theta (\psi-1))\right)\widehat\Pi_1^\theta,
$$
and Agmon estimates on $P_1^\theta$ show the existence of $\delta_2>0$, still independent of $\theta$, such that, for all $j=1,\dots,m$, one has,
$$
\|(1-\psi)u^\theta_j\|_{L^2}={\mathcal O}(e^{-2\delta_2/h}),
$$
and therefore, since $m(h)={\mathcal O}(h^{-n})$,
\be
\label{estA2}
\| \widehat\Pi_1^\theta \left(\chi \Pi_1^\theta \psi)\right)\widehat\Pi_1^\theta\|_{{\mathcal L}(L^2)}={\mathcal O}(e^{-\delta_2/h}).
\ee
From (\ref{estA1})-(\ref{estA2}), we obtain,
$$
\| Y(z)\|_{{\mathcal L}(L^2)}={\mathcal O}(e^{-\delta_3/h}),
$$
for some constant $\delta_3>0$. Going back to (\ref{approxresred}), we deduce,
\be
\label{resredP11}
(\widehat P_1^\theta -z)^{-1}=\widehat\Pi_1^\theta X(z)\widehat\Pi_1^\theta (1+ {\mathcal O}(e^{-\delta_3/h})).
\ee
On the other hand, since the distortion coincides with the identity on the supports of $\chi$ and of $\psi$, we have,
$$
X(z):= \chi (\widehat P_1 -z)^{-1}\psi + (\widetilde P_1^\theta -z)^{-1}(1-\psi),
$$
and by construction $\|(\widetilde P_1^\theta -z)^{-1}\| ={\mathcal O}(1)$ and $\| (\widehat P_1 -z)^{-1}\|={\mathcal O}(1/a)$.
Hence, by (\ref{resredP11}), Lemma \ref{estresredG} follows.

\subsection{Appendix 2: Proof of Lemma \ref{estM0M1}}

In view of (\ref{estQ1Q2}), (\ref{K}), it is enough to prove that, if $A$ is a bounded operator on $L^2(\R^n)$, then the matrix $M_A:= (\la Au_k^\theta,u_j^{-\theta}\ra_{L^2(\R^n)})_{1\leq j,k\leq m}$ satisfies,
\be
\label{L+AL-}
\| M_A\|_{{\mathcal L}(\C^m)} ={\mathcal O}(\| A\|_{{\mathcal L}(L^2)}),
\ee
uniformly with respect to $h>0$ small enough. In order to prove (\ref{L+AL-}), we take $\alpha =(\alpha_1,\dots,\alpha_m)\in \C^m$, and we  write,
$$
\| M_A\alpha \|^2 =\sum_{j=1}^m |\sum_{k=1}^m \alpha_k\la Au_k^\theta,u_j^{-\theta}\ra_{L^2(\R^n)}|^2=\sum_{j=1}^m |\la A\widetilde\alpha ,u_j^{-\theta}\ra_{L^2(\R^{n})}|^2,
$$
where $\widetilde\alpha:=\sum_{k=1}^m \alpha_ku_k^\theta$. Then, we denote by $D\subset \R^n$ and open set such that 
$$
U\subset\, D\subset\, \R^n\backslash {\rm Supp}\, F,
$$
In particular, on $D$ we have $u_k^{\pm \theta} =u_k$, and, by Agmon estimates, we know that the norms $\| u_k^{\pm\theta}\|_{L^2(\R^n\backslash D)}$ are exponentially small, uniformly with respect to $\theta$.  Therefore, since $m={\mathcal O}(h^{-n})$, we can write,
\be
\label{estAnn3}
\| M_A\alpha \|^2=\sum_{j=1}^m |\la A\widetilde\alpha ,u_j\ra_{L^2(D)}|^2 +{\mathcal O}(e^{-c /h})\| A\widetilde \alpha\|_{L^2}^2,
\ee
where $c >0$ is independent of $\alpha$, $\theta$, and $h$. Then, we use the fact that, for the same reason (and since $\la u_k^\theta, u_j^{-\theta}\ra_{L^2}=\delta_{j,k}$), we have,
\be
\label{quasiorth}
\la u_k, u_j\ra_{L^2( D)} =\delta_{j,k}+{\mathcal O}(e^{-c /h}),
\ee
where the positive constant $c$ may be different from the previous one. This permits us to show (e.g., by diagonalizing the family $(u_k)_{1\leq k\leq m}$ in $L^2(D)$ by means of a matrix $B=I+{\mathcal O}(e^{-\delta /h}$)) that one has,
$$
\sum_{j=1}^m |\la A\widetilde\alpha ,Tu_j\ra_{L^2(D)}|^2 ={\mathcal O}(\| A\widetilde\alpha\|_{L^2(D)}^2),
$$
uniformly with respect to $h$ and $\alpha$. Hence, inserting in (\ref{estAnn3}), we find,
$$
\| M_A\alpha \|^2 ={\mathcal O}(\| A\widetilde\alpha\|_{L^2(D)}^2+e^{-c /h}\| A\widetilde \alpha\|_{L^2}^2),
$$
and thus,
$$
\| M_A\alpha \|^2 ={\mathcal O}(\| A\widetilde \alpha\|_{L^2}^2)={\mathcal O}(\| A\|^2\cdot\|\widetilde \alpha\|_{L^2}^2),
$$
and the result follows by observing that (using the decay properties of the $u_k^{\pm\theta}$'s and (\ref{quasiorth}) again),
$$
\|\widetilde \alpha\|_{L^2}={\mathcal O}(\|\widetilde \alpha\|_{L^2(D)}+e^{-c/h}\|\alpha\|_{\C^m})={\mathcal O}(\|\alpha\|_{\C^m}).
$$

{\bf Acknowledgements} A. Martinez was partially supported by the University of Toulon. He is also indebted to the Centre de Physique Th\'eorique de Marseille, where most of this work was done, for his warm hospitality in February and March 2015.
%%%%%%%%%%%%%%%%%%%%%%%%%%%%%%%%%%%%%%%%%%%%%%%%%%%%%%%%%%%%%%%%%

%\section{Distortion}\label{appA}

{}

\end{document}